\documentclass[10pt,halfline,a4paper]{ouparticle}
\usepackage{amsmath}
\usepackage[english]{babel}
\usepackage{amsthm}
\usepackage{IEEEtrantools,stackengine} 
\stackMath 
\usepackage{amsmath}
\usepackage{dsfont}
\usepackage{amssymb}
\usepackage{fixmath}
\usepackage{eucal}
\usepackage{commath} 
\usepackage[T1]{fontenc}
\usepackage{graphicx}
\usepackage{tikz}
\usepackage{pgfplots}
\pgfplotsset{compat=newest}
\pgfplotsset{plot coordinates/math parser=false}
\newlength\figureheight
\newlength\figurewidth
\usepackage{titlesec}
\usepackage[shortlabels]{enumitem}

\usepackage{gauss}

\newtheorem{theorem}{Theorem}[section]
\newtheorem{definition}[theorem]{Definition}
\newtheorem{corollary}[theorem]{Corollary}
\newtheorem{lemma}[theorem]{Lemma}

\theoremstyle{definition}
\newtheorem{example}[theorem]{Example}
\newtheorem{remark}[theorem]{Remark}

\def\N{{\mathbb N}}
\def\Z{{\mathbb Z}}
\def\Q{{\mathbb Q}}
\def\R{{\mathbb R}}

\def\F{{\mathbb F}}

\makeatletter
\newsavebox\myboxA
\newsavebox\myboxB
\newlength\mylenA

\newcommand*\xoverline[2][0.75]{%
    \sbox{\myboxA}{$\m@th#2$}%
    \setbox\myboxB\null
    \ht\myboxB=\ht\myboxA%
    \dp\myboxB=\dp\myboxA%
    \wd\myboxB=#1\wd\myboxA
    \sbox\myboxB{$\m@th\overline{\copy\myboxB}$}
    \setlength\mylenA{\the\wd\myboxA}
    \addtolength\mylenA{-\the\wd\myboxB}%
    \ifdim\wd\myboxB<\wd\myboxA%
       \rlap{\hskip 0.5\mylenA\usebox\myboxB}{\usebox\myboxA}%
    \else
        \hskip -0.5\mylenA\rlap{\usebox\myboxA}{\hskip 0.5\mylenA\usebox\myboxB}%
    \fi}
\makeatother

\newcommand{\PPFp}{\mathbb{F}_p((t^{1/p^{\infty}}))}

\newcommand{\FpbarQ}{\bar{\mathbb{F}}_p((t^{\mathbb{Q}}))}
\newcommand{\FbarQ}{\bar{\mathbb{F}}((t^{\mathbb{Q}}))}

\newcommand{\FpQ}{\mathbb{F}_p((t^{\mathbb{Q}}))}

\newcommand{\Fp}{\mathbb{F}_p}

\newcommand{\Fpbar}{\bar{\mathbb{F}}_p}
\newcommand{\Fbar}{\bar{\mathbb{F}}}

\newcommand\maxram{\textsc{maximal\_ramification}}
\newcommand\maxexp{\textsc{maximal\_expansion}}
\newcommand\addpol{\textsc{additive\_multiple}}

\title{Approximation and algebraicity in positive characteristic Hahn fields}
\author{Victor Lisinski\thanks{The author was funded by an EPSRC award at the University of Oxford, with additional support from the Royal Swedish Academy of Sciences and Corpus Christi College Oxford.}}

\date{}

\newcommand{\Addresses}{{
  \vspace*{3em}
  \footnotesize

\textsc{Mathematical Institute, Woodstock Road, Oxford OX2 6GG.}\par\nopagebreak
  \textit{E-mail address}: \texttt{lisinski@maths.ox.ac.uk}
}}

\begin{document}

\abstract{
We study the relative algebraic closure $K$ of $\Fpbar((t))$ inside $\FpbarQ$. We show that the supports of elements in $K$ have order type strictly less than $\omega^\omega$. We also recover a theorem by Rayner giving a bound to the ramification away from $p$ in the support of elements in $K$, and an analogue of Rayner's result for the residue field. This work has applications to the decidability of the first order theory of $\FpQ$, and other tame fields, in the language of valued fields with a constant symbol for $t$.}

\maketitle

\section{Introduction}
\label{sec: intro}
    
It was shown by Puiseux, and implicitly apparent in work by Newton, that when $\F$ has characteristic $0$ then the algebraic closure of $\Fbar((t))$ is the direct limit of $(\Fbar((t^{1/n})))_{n\in\N}$ \cite[p. 295]{MR1322960}. Another way to formulate this fact is by the following two properties about the relative algebraic closure $K$ of $\Fbar((t))$ in the Hahn field $\FbarQ$:
\begin{enumerate}
    \item The support of elements in $K$ have order type at most $\omega$.
    \item Elements in $K$ cannot have infinite ramification in the support, in the sense made more precise in Definition \ref{def: ramification in element}.
\end{enumerate}
Note that the first property implies that no element in $K$ can have bounded support, since this would give an element with support of order type $\omega+1$ by adding a suitable power of $t$. Both these properties are contradicted in characteristic $p$, for example by the generalised power series $1+\sum_{n\in\N}t^{-1/p^n}$, which is a root to the Artin-Schreier polynomial $X^p-X-1/t$, as shown in \cite{MR80647}. In this paper we use a transfinite approximation method to obtain the following modified properties in positive characteristic:
\begin{enumerate}
    \item The support of elements in $K$ have order type strictly smaller than $\omega^\omega$.
    \item Elements in $K$ cannot have infinite ramification away from $p$ in the support.
\end{enumerate}
The first item seems to not have been treated in the literature, though the order type of the support of elements in Hahn fields of characteristic zero is studied in \cite{MR3911738}. The second item is originally due to Rayner. Our method provides an alternative proof which also carries over to an analogue statement about the residue field. This is used in \cite{lisinski} to show that the first order theory of $\F((t^\Gamma))$ is decidable in the language of valued fields with a constant symbol for $t^\gamma$, with $\gamma\in\Gamma_{>0}$, when $\F$ is a perfect field of characteristic $p$ which is decidable in the language of rings and $\Gamma$ is a $p$-divisible ordered abelian group which is decidable in the language or ordered groups with a constant symbol for $\gamma$.

\section{Notation and preliminaries}
\label{sec: notation}
\begin{itemize}
\item For a field $K$, we denote its algebraic closure by $\bar{K}$.
\item For an ordered abelian group $(\Gamma,<)$ and an element $r\in\Gamma$, we write $\Gamma_{<r}=\{\gamma\in\Gamma \ | \ \gamma<r\}$ and $\Gamma_{\le r}=\Gamma_{<r}\cup\{r\}$.
\item For a field $K$ and an ordered abelian group $\Gamma$, we denote by $K((t^\Gamma))$ the Hahn field consisting of generalised power series on the form $x=\sum_{\gamma\in S_x}a_\gamma t^\gamma$, with $S_x$ a well ordered subset of $\Gamma$ and $a_\gamma\in K$. We write $K((t))$ instead of $K((t^\Z))$.
\item A Hahn field $K((t^\Gamma))$ is a valued field with the $t$-adic valuation, denoted by $v$, sending $\sum_{\gamma\in S}a_\gamma t^\gamma$ to the minimal element in the set $\{\gamma \ | \ a_\gamma\neq 0\}$. Such a minimum exists since $S$ is well ordered.
\item For a generalised power series $x=\sum_{\gamma\in\Gamma}a_\gamma t^\gamma$, we will interchangeably use the notations
\begin{align*}
x=\sum_{\gamma\ge\gamma_0}a_\gamma t^{\gamma}, \hspace{1.5em}  x=\sum_{i< \sigma}a_i t^{\gamma_i}    
\end{align*}
where $\gamma_0=v(x)$ and $\sigma$ is an ordinal. For $r\in\Gamma$, we write
\begin{align*}
x_{<r}=\sum_{\gamma\in\Gamma_{<r}}a_\gamma t^\gamma, \hspace{1.5em} x_{\le r}=\sum_{\gamma\in\Gamma_{\le r}}a_\gamma t^\gamma.
\end{align*}
Similarly, with the notation $x=\sum_{i< \sigma}a_i t^{\gamma_i}$, we write $x_{<i}=x_{<\gamma_i}$ and $x_{\le i}=x_{\le \gamma_i}$. 
\end{itemize}
\begin{definition}
We say that $y$ is an \textbf{approximation} of $x$ if $y=x_{<\lambda}$ for some $\lambda$ and $y\neq x$. We say that $y$ is a \textbf{better approximation} of $x$ than $z$ if $y$ is an approximation of $x$ and $z$ is an approximation of $y$.
\end{definition}

\begin{definition}
Let $K$ be a field and let $P(X)\in K[X]$. Then $P$ is called \textbf{additive} if $P(X+Y)=P(X)+P(Y)$ as polynomials in $X$ and $Y$. 
\end{definition}

Our use of additive polynomials relies completely on the following important result. It appears for example as Lemma 12.2.3 in \cite{MR1997038}.

\begin{lemma}[Ore's lemma]
\label{lem: ore}
Let $K$ be a field of characteristic $p$ and let $f(X)\in K[X]$. Then there is a non-zero additive polynomial $P(X)\in K[X]$ of degree $p^{\deg(f)}$ such that $f$ divides $P$.
\end{lemma}

\begin{proof}
Consider $\{X^{p^i} \mod f(X) \ | \ 0\le i\le \deg(f) \}$ as a set of vectors in the $K$-vector space $K[X]/(f)$. By cardinality, these are linearly dependent over $K$. Thus, there are $a_i\in K$ such that
\[ \pushQED{\qed} \sum_{i=0}^{\deg(f)}a_iX^{p^i}\equiv 0 \mod f(X). \qedhere \popQED\]
\renewcommand{\qedsymbol}{}
\vspace{-\baselineskip}
\end{proof}

\begin{remark}
\label{rem: ore alg}
In the case of Ore's lemma when $K=\Fp(t)$, there is an algorithm that takes as input $f$ and returns the additive polynomial $P$. We denote this algorithm by $\addpol$. To see that this algorithm indeed exists, we note that we can use Euclid's algorithm to obtain $b_{k,i}\in \Fp(t)$ for $k<\deg(f)$ and $i\le \deg(f)$ such that
$$X^{p^i} \equiv \sum_{k=0}^{\deg(f)-1} b_{k,i}X^k.$$
We then let $(a_i)_{0\le i\le \deg(f)}$ be a non-trivial solution to the system of linear equations over $\F_p$ given by
$$\sum_{i=0}^{\deg(f)} a_i\sum_{k=0}^{\deg(f)-1}b_{k,i} X^k  \equiv 0 \mod f.$$
On input $f(X)$, the algorithm $\addpol$ thus returns $\sum_{i=0}^{\deg(f)}a_iX^{p^i}$.
\end{remark}

\section{Approximating roots}
\label{sec: approximation}
    
This section builds on a transfinite recursion procedure introduced in \cite{MR846958} for constructing roots in Hahn fields to polynomials over the underlying field of formal Laurent series. For this, we need a version of Taylor expansions in positive characteristic.

\begin{definition}
Let $K$ be a field and let
$$f(X)=\sum_{i=0}^n a_iX^i \in K[X].$$
Then the $k$-th Hasse derivative of $f$ is defined as
$$D^{(k)}f(X) := \sum_{i=k}^n  \binom{i}{k} a_i X^{i-k} \in K[X].$$
\end{definition}

\begin{remark}
    We will not consider derivatives of formal power series. Thus, with $K$ being a field of power series, $f\in K[X]$ and $x\in K$, we will write $D^{(k)}f(x)$ in place of $(D^{(k)}f)(x)$. We will also omit parenthesis in expressions of the form $(D^{(k)}f(x))g(X)$, since we will never consider expressions of the form $D^{(k)}(f(x)g(X))$.
\end{remark}

\begin{theorem}[Taylor expansion]
Let $K$ be a field and let $f(X)\in K[X]$ be of degree $n$ and let $\lambda \in K$. Then
$$f(X) = \sum_{k=0}^{n}D^{(k)}f(\lambda)(X-\lambda)^k.$$
\end{theorem}
    
\begin{proof}
By linearity of the Hasse derivative, it is enough to show the statement for $f(X)=X^n$. We have
    
\begin{IEEEeqnarray*}{LL}
\sum_{k=0}^{n}D^{(k)}f(\lambda)(X-\lambda)^k &= \sum_{k=0}^n \binom{n}{k}\lambda^{n-k}(X-\lambda)^k \\
    &= ((X-\lambda)+\lambda)^n\\
    &=X^n,
\end{IEEEeqnarray*}
and we are done.
\end{proof}
    
To fix some notation, let $\F$ be an algebraically closed field of characteristic $p$, let $\Gamma$ be a divisible ordered abelian group, and let
$$f(X)=\sum_{i=0}^n a_iX^i \in \F((t^\Gamma))[X].$$
With $x \in \F((t^\Gamma))$ and $i\in\{1,\ldots,n\}$, we define 
$$g_{i,x}(X):=D^{(i)}f(x)X^i.$$
Let $b_{i,x}t^{\rho_{i,x}}$, with $b_{i,x}\in\F$, denote the initial term of $D^{(i)}f(x)$. To clarify the relationship between $g_{i,x}$ and $b_{i,x}t^{\rho_{i,x}}$, we have that $g_{i,x}(X)=(b_{i,x}t^{\rho_{i,x}}+z)X^i$, where $z\in\F((t^\Gamma))$ is of valuation strictly greater than $\rho_{i,x}$.

We now define the function $\gamma_{i,x} : \Gamma\cup\{\infty\}\to \Gamma\cup\{\infty\}$ by letting $\gamma_{i,x}(r)=\rho_{i,x}+ir$ for $r\in\Gamma$ and $\gamma_{i,x}(\infty)=\infty$. In particular, we get that $\gamma_{i,x}(r)=vg_{i,x}(y)$ for any $y\in\F((t^\Gamma))$ of valuation $r$. Let $\gamma_x(r):=\min_{i\in\{1,\ldots,n\}}\left\{\gamma_{i,x}(r)\right\}$ and let $J_x(r):=\left\{i\in\{1,\ldots,n\} \ | \ \gamma_{i,x}(r) = \gamma_x(r)\right\}$. By construction, we have that $vg_{i,x}(y)>\gamma_x(r)$ for any $i\notin J_x(r)$. This implies that
$$v\left(\sum_{i=1}^ng_{i,x}(\zeta t^r)\right)\ge \gamma_x(r),$$
with equality if and only if $\sum_{i\in J_x(r)} b_{i,x}\neq 0$.

Suppose now that $x$ is a root of $f$ and consider an approximation $w=x_{< \lambda}$ of $x$. We want to find $r\in\Gamma$ and $\zeta\in \F$ such that $w+\zeta t^r$ is a better approximation of some root $\alpha\in\F((t^\Gamma))$ of $f$ such that $\alpha_{<\lambda}=w$. To do this, we will find $\zeta$ and $r$ such that $vf(w)<vf(w+\zeta t^r)$, and then proceed by transfinite recursion. Let $r$ be such that $\gamma_w(r)=vf(w)$. For any $\zeta\in\F$, we have the equality
$$f(w+\zeta t^r) = f(w)+\sum_{i=1}^ng_{i,w}(\zeta t^r).$$
If $bt^{\gamma_w(r)}$ is the initial term of $f(w)$, we let $\zeta$ be such that 
\begin{align}
\label{eqn: approximation term}
    \sum_{i\in J_w(r)}b_{i,w}\zeta^i=-b
\end{align}
to get that $vf(w+\zeta t^r)>\gamma_w(r)$, as we wanted. We capture this procedure in the following definition, using the same notation as above.

\begin{definition}
\label{def: approximation term}
Let $f(X)\in\F((t^\Gamma))[X]$ and let $w\in\F((t^\Gamma))$ be the approximation of a root of $f$. We say that $\zeta t^r$ is an \textbf{approximation term} for $w$ with respect to $f$ if $w_{<r}=w$, if $\gamma_w(r)=vf(w)$, and if $\zeta\in\F$ satisfies the equality \textup{(\ref{eqn: approximation term})}.
\end{definition}

As mentioned, successively adding approximation terms eventually gives a root to $f$. More precisely, we define a transfinite recursion procedure as follows. Let $w_0 := w$. For any ordinal $\sigma> 0$ such that $f(w_\lambda)\neq 0$ for all $\lambda<\sigma$, let $$w_\sigma = w + \sum_{\lambda<\sigma}\zeta_\lambda t^{r_\lambda}$$
where $\zeta_\lambda t^{r_\lambda}$ is a fixed approximation term for $w_\lambda$ with respect to $f$. We then get the following lemma, which appears implicitly in \cite{MR846958}.

\begin{lemma}
\label{lem: approximation term converges}
There is an ordinal $\sigma$ such that $f(w_\sigma)=0$.
\end{lemma}

\begin{proof}
It is enough to show that $vf(w_\sigma)>vf(w_\tau)$ for all ordinals $\sigma$ and $\tau$ such that $\sigma>\tau$. Indeed, if this inequality holds but no $w_\sigma$ is a root of $f$, then by taking $\sigma$ to be an ordinal of cardinality strictly greater than $\lvert\Gamma\rvert$, we get that $\left(vf(w_\lambda)\right)_{\lambda\le\sigma}$ is a strictly increasing sequence of cardinality strictly greater than $\lvert\Gamma\rvert$ in $\Gamma$. 

Let $\sigma>\tau$. We note that
$$w_\sigma-w_\tau=w+\sum_{\lambda<\sigma}\zeta_\lambda t^{r_\lambda}-\left(w+\sum_{\lambda<\tau}\zeta_\lambda t^{r_\lambda}\right)=\sum_{\tau\le \lambda <\sigma}\zeta_\lambda t^{r_\lambda}.$$
Therefore, we get
$$f(w_\sigma)=f\left(w_\tau+\sum_{\tau\le \lambda<\sigma}\zeta_\lambda t^{r_\lambda}\right)=f(w_{\tau+1})+\sum_{i=1}^ng_{i,w_{\tau+1}}\left(\sum_{\tau< \lambda<\sigma}\zeta_\lambda t^{r_\lambda}\right).$$
Since $vf(w_{\tau+1})>vf(w_\tau)$ by construction, and since 
$$v\left(\sum_{i=1}^ng_{i,w_{\tau+1}}\left(\sum_{\tau< \lambda<\sigma}\zeta_\lambda t^{r_\lambda}\right)\right)\ge \gamma_{w_{\tau+1}}(r_{\tau+1})=vf(w_{\tau+1}),$$
we conclude that $vf(w_\sigma)>vf(w_\tau)$.
\end{proof}

\begin{remark}
\label{rem: around zero}
Even if $w$ is not the approximation of a root of $f$, we can still use the transfinite recursion procedure to find a root to the polynomial $f(w+X)$. Hence, we obtain a sequence $(\zeta_\lambda t^{r_\lambda})_{\lambda<\sigma}$ with $r_\lambda$ strictly increasing such that $w+\sum_{\lambda<\sigma}\zeta_\lambda t^{r_\lambda}$ is a root of $f$. In particular, we find the possible initial terms of roots of $f$ as the possible approximation terms for $0$ with respect to $f$, i.e. by finding $\zeta t^r$ such that the initial term of $\sum_{k=1}^nD^{(k)}f(0)(\zeta t^r)^k$ cancels the initial term of $f(0)$.
\end{remark}

While successively adding approximation terms gives a root of $f$, we will see that it is not a sufficient procedure to recover all the roots. Using the same notation as above, and with $x$ being a root of $f$, suppose now that $r\in\Gamma$ is such that $x_{\le r}$ is a better approximation of $x$ than $w$, as defined in the end of Section \ref{sec: notation}. In particular, if $\zeta\in \F$ is such that $x_{\le r}=w+\zeta t^r$, then $\zeta\neq 0$. In this situation, we can write $x=w+\zeta t^r+y$, with $v(y)>r$. We have the equality
    $$f(w+\zeta t^r+y) = f(w)+\sum_{i=1}^ng_{i,w}(\zeta t^r+y)^i.$$
Since $v(y)>r$, we have that $b_{i,w}\zeta^it^{\rho_{i,w}+ir}$ is the initial term of both $g_{i,w}(\zeta t^r)^i$ and $g_{i,w}(\zeta t^r+y)^i$. In particular, we have $vg_{i,w}(\zeta t^r)=vg_{i,w}(\zeta t^r+y)=\gamma_{i,w}(r)$.

\begin{align}
\label{val equality in recursion}
vf(w)=v\left(\sum_{i=1}^ng_{i,w}(\zeta t^r+y)\right).
\end{align}
In particular, we have
\begin{align}
\label{val inequality in recursion}
vf(w) \ge \gamma_w(r).
\end{align}
If the inequality in (\ref{val inequality in recursion}) is strict, then equality (\ref{val equality in recursion}) gives
$$v\left(\sum_{j\in J_w(r)}g_{j,w}(\zeta t^r+y)\right)>\gamma_w(r).$$
In particular, this holds if $f(w)=0$. For any $j\in J_w(r)$, the term of valuation $\gamma_w(r)$ in $g_{j,w}(\zeta t^r+y)$ is simply the initial term, i.e. $b_{j,w}\zeta^jt^{\gamma_{j,w}(r)}$. Furthermore, by deinifion of $J_w(r)$, we have that this term in fact is equal to $b_{j,w}\zeta^jt^{\gamma_w(r)}$. Hence, we conclude that $\zeta$ satisfies
\begin{align}
    \label{first condition for zeta}
    \sum_{j\in J_w(r)}b_{j,w}\zeta^j=0.
\end{align}
Consequently, $\lvert J_w(r)\rvert>1$, since we otherwise would have $\zeta=0$. Assume now that (\ref{val inequality in recursion}) is an equality. Then, we are exactly in the situation preciding Definition \ref{def: approximation term}, and $\zeta t^r$ is an approximation term for $w$ with respect to $f$, i.e. $\zeta$ satisfies \ref{eqn: approximation term}.

We summarise the discussion in the following lemma.

\begin{lemma}
\label{lem: necessary approximation condition}
Let $f(X)\in \F((t^\Gamma))$. If $w+\zeta t^r$ is a better approximation than $w$ of some root $x$ of $f$, then
\begin{enumerate}[(a)]
    \item $\zeta$ satisfies (\ref{first condition for zeta}) and $\lvert J_w(r)\rvert>1$ or
    \item $\zeta t^r$ is an approximation term for $w$ with respect to $f$.
\end{enumerate}
\end{lemma}
    
In some non-trivial cases, we can use approximation terms to directly determine if a polynomial has a root in a particular field or not, as illustrated by the following example.
\begin{example}
\label{ex: recursion says all}
Let
$$f(X)=X^3-X^2-\frac{1}{t}\in\mathbb{F}_3(t)[X].$$
We will show by induction that there is a root
$$x\in\xoverline{\mathbb{F}}_3((t^{\mathbb{Q}}))$$
such that each term of $x$ is of the form $a t^\gamma$ with $a\in\F_3$ and $\gamma \in \frac{1}{3^\infty}\Z$, i.e. such that $x\in \F_3((t^{1/3^{\infty}}))$. For any such root, we  have by Remark \ref{rem: around zero} that the initial term is equal to $t^{-1/3}$. Assume that we have an approximation $w=x_{<\lambda}\in \mathbb{F}_3((t^{1/p^{\infty}}))$ for some ordinal $\lambda$. If $f(w)=0$ we are done, so assume $f(w)\neq 0$. Let $\zeta t^r$ be an approximation term for $w$ with respect to $f$. Then, we have that $vf(w)= \min\left\{r-1/3, \ 2r, \ 3r\right\}$. By definition of approximation term, we also have that $r>v(w)$. So $2r>r-1/3$ and $vf(w)\neq 2r$. Suppose that $r-1/3=3r$. Then $r=-1/6$, so $vf(w)=-1/2$. This is a contradiction, since $w\in \F_3((t^{1/3^{\infty}}))$. Hence, we have that $vf(w)$ is equal to either $r-1/3$ or to $3r$, but not both.

Let $\xi t^{vf(w)}$, with $\xi \in \F_3$.  If $vf(w)=r-1/3$, we let $\zeta=\xi/2$. If $vf(w)=3r$, we set $\zeta^3=-\xi$, i.e. $\zeta=-\xi$. In both cases, $\zeta t^r\in\F_3(t)^{1/3^\infty}$. This shows that each approximation in the recursion procedure preceding Lemma \ref{lem: approximation term converges} lies in $\F_3((t^{1/3^{\infty}}))$, so we conclude that the root $w_\sigma$ in the lemma also lies in $\F_3((t^{1/3^{\infty}}))$.

In this example, we can also use Lemma \ref{lem: necessary approximation condition} to determine the minimal Hahn field containing all roots of $f$. Consider the following functions, as in the discussion preceding Defintion \ref{def: approximation term} recalling that $v(w)=-1/3$ for any non-zero approximation $w$ of a root $x$ of $f$.
\begin{align*}
    \gamma_{1,w}(r) & =-1/3r; \\
    \gamma_{2,w}(r) & =2r; \\
    \gamma_{3,w}(r) & =3r.
\end{align*}
Since these function do not depend on $w$, we will just write $\gamma_1$, $\gamma_2$ and $\gamma_3$. Let $\zeta t^r$ be such that $w=\zeta t^r$ is a better approximation than $w$ of some root $x$ of $f$. Assuming that $\zeta t^r$ is not an approximation term for $w$ with respect to $f$, we get by Lemma \ref{lem: necessary approximation condition} that $\lvert J_w(r)\rvert>1$. Pairs of the lines defined by the $\gamma_i$ intersect at $r=-1/6$, $r=-1/3$ and $r=0$. For $r=-1/3$ and $r=0$, we have that $\lvert J_w(r)\rvert=1$. For $r=-1/6$ however, we have that $J_w(r)=\{1,3\}$. Since $\zeta$ satisfies
$$-2\zeta+\zeta^3=0,$$
we have that $\zeta\in\{\sqrt{2},\sqrt{2}\}$. Therefore, $w+\zeta t^r\in K:=\F_3(\sqrt{2})((t^{\frac{1}{2\cdot3^\infty}}))$. Repeating the argument above for the existence of a root in $\F_3((t^{1/3^\infty}))$ with $w'=w+\zeta t^r$ in place of $w$ shows that $f$ has a root $w'_\sigma\in K$.

Since the choice of approximation term for valuation $r$ for $w$ is unique when $\lvert J_w(r)\rvert =1$, and since we showed that there was no approximation term with valuation $r=-1/6$, we get that $w_\sigma$ and $w'_\sigma$ are the unique roots $x$ and $y$ such that $x_{\le -1/6}=w$ and $y_{\le -1/6}=w+\zeta t^r$. Since any root of $f$ in $\Fbar_3((t^\Q))$ must have $w$ as an initial sum, we get that the third root $\alpha$ of $f$ in $\Fbar_3((t^\Q))$ must satisfy $\alpha_{\le r}=w-\zeta t^r$. Hence, $\alpha$ is obtained by repeating the recursion procedure for $w-\zeta t^r$, which shows that $\alpha\in K$.  Conversely, any Hahn field containing $w'_\sigma$ and $\alpha$ will contain $K$ as a subfield. 
\end{example}

    %
    
    
\begin{remark}
The reason why the recursion procedure cannot be used immediately for arbitrary polynomials is that it can be difficult to compute $D^{(n)}f(x_{<\lambda})$. More precisely, since the truncations of $x_{<\lambda}$ give a pseudo convergent sequence of algebraic type, we might have that $vD^{(n)}f(x_{<\tau})< vD^{(n)}f(x_{<\lambda})$ for any $\tau<\lambda$, so $D^{(n)}f(x_{<\lambda})$ is not computable by finite approximation. Note that Krasner's Lemma does not work in this situation, since the value of $D^{(n)}f(x_{<\tau})$ might be bounded. However, if the Hasse derivatives are constant, the situation becomes easier, as we will see in the special case of additive polynomials.
\end{remark}

\section{Order type of the support of algebraic elements}
\label{sec: order type}
When $K$ has characteristic $p$, then $P\in K[X]$ is additive if and only if it is of the form
$$\sum_{i=0}^n a_iX^{p^i}.$$
In this case, the $D^{(p^k)}P(X)=a_kX^{p^k}$ and $D^{(\ell)}P(X)=0$ for $\ell$ not a power of $p$. Since these Hasse derivatives are constant, we get in particular that the functions $\gamma_{i,w}$ and the set $J_w(r)$ in Section \ref{sec: approximation} do not depend on $w$. When $P$ is given, we will thus only write $\gamma_i$ and $J(r)$. Furthermore, if $w$ is the approximation of a root of $P$, using the Taylor approximation of $P$ to find an approximation term $\zeta t^r$ for $w$ with respect $P$ just amounts to writing
$$P(w+ \zeta t^r) = P(w)+P(\zeta t^r).$$
As we will see, this implies that the material in Section \ref{sec: approximation} has an elementary geometric interpretation which simplifies computation. This is made possible with the following definition.
    
\begin{definition}
\label{def: pt of intersection}
Let $\F$ be a field of characteristic $P$ and let $\Gamma$ be an ordered abelian group. Let $P(X)\in\F((t^\Gamma))[X]$ be an additive polynomial. We say that $r\in\mathbb{R}\cup\{\infty\}$ is a \textbf{point of intersection of $P$} if $\lvert J(r)\rvert >1$.
\end{definition}
    
\begin{remark}
\label{rem: number of intersection points}
The number of points of intersection of $P$ is bounded by the maximal number of intersection points of $n+1$ lines, including the point at infinity, i.e. by $n(n+1)/2+1$. Furthermore, there is an algorithm which takes as input an additive polynomial $P(X)\in\F_p[t][X]$ and outputs the points of intersection of $P$. This can be seen simply by noting that the $\gamma_i$ in this case are lines with integer slopes.
\end{remark}
    
Let $P$ be an additive polynomial, let $x$ be a root of $P$ with value $r_0$, and let $r$ be the minimal point of intersection $P$ larger than $r_0$. With Definition \ref{def: pt of intersection} in mind, the recursion procedure of approximating a root of $P$ via approximation terms can be illustrated as follows.
    
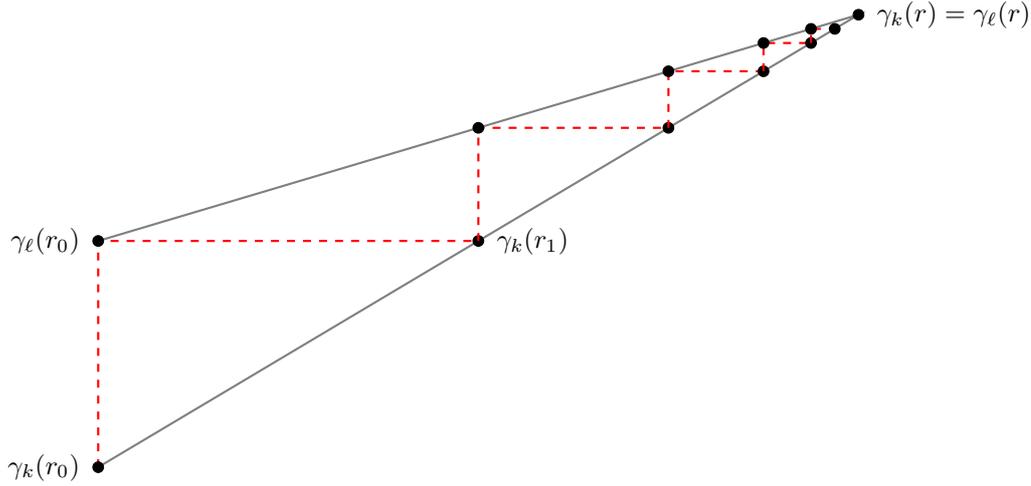
\begin{tikzpicture}
\draw[gray, thick] (0,0) -- (10,6) -- (0,3);
\draw[red, dashed, thick] (0,0) -- (0,3) -- (5,3);
\filldraw[black] (10,6) circle (2pt) node[anchor=west] {\ $\gamma_k(r)=\gamma_\ell(r)$};
\filldraw[black] (0,0) circle (2pt) node[anchor=east] {$\gamma_k(r_0)$ \ };
\filldraw[black] (5,3) circle (2pt) node[anchor=west] {\ $\gamma_k(r_1)$};
\filldraw[black] (0,3) circle (2pt) node[anchor=east] {$\gamma_\ell(r_0)$ \ };
    
\draw[red, dashed, thick] (5,3) -- (5,4.5) -- (7.5,4.5);
\filldraw[black] (5,4.5) circle (2pt);
\filldraw[black] (7.5,4.5) circle (2pt);
    
\draw[red, dashed, thick] (7.5,4.5) -- (7.5,5.25) -- (8.75,5.25);
\filldraw[black] (7.5,5.25) circle (2pt);
\filldraw[black] (8.75,5.25) circle (2pt);
    
\draw[red, dashed, thick] (8.75,5.25) -- (8.75,5.625) -- (9.375,5.625);
\filldraw[black] (8.75,5.625) circle (2pt);
\filldraw[black] (9.375,5.625) circle (2pt);
    
\draw[red, dashed, thick] (9.375,5.625) -- (9.375,5.8125) -- (9.6875,5.8125);
\filldraw[black] (9.375,5.8125) circle (2pt);
\filldraw[black] (9.6875,5.8125) circle (2pt);
    
\end{tikzpicture}
    
The illustration suggests the procedure converges to $r$, so that the minimal element in the support of $x$ larger than or equal to $r$ would be $r_\omega$. This is an instance of a more general fact which is captured in the following lemma.
\begin{lemma}
\label{lem: order increasing at intersection}
Let 
$$P(X)=\sum_{i=0}^m a_iX^{p^i}\in\F_p((t))[X]$$
be an additive polynomial and let
$$x=\sum_{\tau<\sigma}b_\tau t^{r_\tau}\in\FpbarQ$$
be a root of $P$ with support $\{a_\tau \ | \ \tau<\sigma\}$. Let $\lambda<\sigma$ be such that $\sigma\ge\lambda\cdot\omega$. Then, $\lim_{i\to\infty}r_{\lambda\cdot i}$ is a point of intersection of $P$.
\end{lemma}
    
\begin{proof}
Let $r=\sup_{i\in \N}\{r_{\lambda\cdot i}\}\in\R\cup\{\infty\}$ and let $w=x_{<r}$. While $r$ might not be in $\Q$, it makes sense to talk about the set $J(r)$, by extending the functions $\gamma_i$ to $\R\cup\{\infty\}$ in the natural way. Let $(s_i)_{i\in \N}$ be a subsequence of $(r_{\lambda\cdot i})_{i\in \N}$ such that $x_{\le s_i}-x_{< s_i}$ is an approximation term for $x_{<s_i}$ with respect to $P$, for every $i$. Such a subsequence exists, since $\lvert J(r_\tau)\rvert>1$ except when $r_\tau$ is a point of intersection of $P$. In particular, we have that $\lim_{i\to\infty}s_i=r$. Let $b'_i$ be such that $x_{\le s_i}-x_{< s_i}=b'_i t^{s_i}$. By definition, we then have that $vP(b'_it^{s_i})=vP(x_{< s_i})$. Let $\ell_i$ and $k_i$ be such $vP(b'_it^{s_i})=\gamma_{\ell_i}(s_i)$ and $\gamma_{k_i}=vP(x_{<s_i})$. In particular, $\ell_i\neq k_i$. We can assume that neither $\ell_i$ nor $k_i$ depend on $i$, possibly by replacing $(s_i)$ with a subsequence. Write $\ell=\ell_i$ and $k=k_i$. Since $s_i$ converges to $r$, we can let $M\in \N$ be such that $\gamma_j(s)>\gamma_\ell(r)$ for all $j\notin J(r)$ and all $s>s_M$. Similarly, we can let $N\in \N$ be such that $\gamma_\ell(s_n)>\gamma_j(s_M)$ for all $n>N$ and all $j\notin J(r)$. Now, let $n>N$ and let $s<r$ be such that $\gamma_k(s)=\gamma_\ell(s_n)$. If $k\notin J(r)$, then $\gamma_\ell(s_n)>\gamma_k(s_M)$. So $s>s_M$, but then $\gamma_k(s)>r$, which contradicts $\gamma_k(s)=\gamma_\ell(s_n)$. Hence, $k\in J(r)$ and $r$ is a point of intersection of $P$.
\end{proof}
    
\begin{remark}
\label{rem: more general than additive}
Note that Lemma \ref{lem: order increasing at intersection} only uses that the Hasse derivatives of $P$ are constant. One could thus formulate a similar statement for more general polynomials where this holds.
\end{remark}
    
    We can now determine a bound for the order type of elements in $\FpbarQ$ that are algebraic over $\F_p((t))$.
    
    \begin{theorem}
    \label{thm: algebraic order type}
    Let $x\in\FpbarQ$ be algebraic over $\F_p((t))$. Then $x$ has order type at most $\omega^m$, where $m=n(n+1)/2+1$ and $n$ is the degree of $x$ over $\F_p((t))$.
    \end{theorem}
    
    \begin{proof}
    Let $f$ be the minimal polynomial of $x$ over $\F_p((t))$ and let $P(X)=\sum_{k=0}^n a_kX^{p^k}$ be an additive polynomial divisible by $f$, as in Lemma \ref{lem: ore}. By induction and Lemma \ref{lem: order increasing at intersection}, the order type of $x$ is at most $\omega^m$, where $m$ is the number of points of intersection of $P$. As noted in Remark \ref{rem: number of intersection points}, $m$ is bounded above by $n(n+1)/2+1$, which proves the theorem.
    \end{proof}
    
    \begin{remark}
    The idea with Theorem \ref{thm: algebraic order type} is give a general bound to algebraic elements. Given $f$, we can obtain a sharper bound by considering the actual points of intersection of the corresponding additive polynomial.
    \end{remark}
    
    \section{Bounds on the ramification away from $p$}
    \label{sec: rayner}
    \begin{definition}
    \label{def: ramification in element}
    Let $x\in\F((t^\Gamma))$. Let $r\in\Gamma$ be in the support of $x$. Let $S$ be the support of $x_{<r}$. We say that a prime $q$ \textbf{ramifies} at $r$ in $x$  if there exists a positive integer $K$ such that $r=\frac{a}{bq^K }$ and, for any $k\ge K$ there is no element $s=\frac{a'}{b'q^k}\in S$, where $a$, $a'$, $b$ and $b'$ are all coprime to $q$. When $x$ is clear from context, we will just say that $q$ ramifies at $r$. 
    \end{definition}
    
    \begin{definition}
    Let $x\in\F((t^\Gamma))$. Let $r\in\Gamma$ and suppose that $x_{\le r}-x_{<r}=\zeta t^r$ is non-zero. Let $K$ be the minimal subfield of $\F$ such that $x_{<r}\in K((t^\Gamma))$. We say that $\zeta$ \textbf{expands} $x$ at $r$ if $\zeta\notin K$.
    \end{definition}
    
    As mentioned, there are elements in the relative algebraic closure of $\F_p(t)$ in $\PPFp$ that require infinite ramification for $p$. It was established by Rayner that this is not the case for primes different than $p$ \cite{MR234941}. We will give a new proof of this result using additive polynomials. This method will also give an effective bound for the ramification away from $p$, which is used in \cite{lisinski}. Similarly, we will show that expanding away from $p$:th roots is bounded, in a sense made more precise below. 
    
    Let $P(X)=\sum_{i=0}^n a_iX^{p^i}\in\F_p(t)[X]$ be an additive polynomial and write
    $$I_P:=\{i\in\{0,\ldots, n\} \ | \ a_i\neq 0\}.$$
    As above, we denote by $\gamma_i$ the function on $\Q\cup\{\infty\}$ sending $r$ to $p^ir+v(a_i)$ for $i\in I_P$.
    
    \begin{theorem}
    \label{thm: ramifies at intersection}
    Let $\F$ be a field of characteristic $p$ and let $P(X)\in\F_p(t)[X]$
    be an additive polynomial. Let $q$ be a prime different from $p$. If $x\in\F((t^\Q))$ is a root of $P$ and $r\in\Q$ is such that $q$ ramifies at $r$ in $x$, then $r$ is a point of intersection of $P$.
    \end{theorem}
    
    \begin{proof}
    Write $P=\sum_{i=0}^n a_iX^{p^i}$. Let $r\in\Gamma$ be in the support of $x$ and suppose that $q\neq p$ is a prime which ramifies in $x$ at $r$. 
    Let $\zeta t^r=x_{\le r}-x_{<r}$. Let $c_1t^{\gamma_1}$ be the initial term of $P(x_{<r})$ and let $c_2t^{\gamma_2}$ be the initial term of $P(\zeta t^r)$. If $\zeta t^r$ is an approximation term for $x_{<r}$, we have that $\gamma_\ell(r)=\gamma_2$ for some $\ell\in J(r)$. This implies that
    $$r=\frac{\gamma_2-v(a_\ell)}{p^\ell}.$$
    By the assumption on $q$, this is only possible if
    $$\gamma_2-v(a_\ell)=\frac{a}{bq^k}$$
    for some $a,b\in\Z$ coprime to $q$ and some positive integer $k$. On the other hand, since $P$ is additive, there is some $j\in I_P$ and some $r_\tau$ in the support of $x_{<r}$ such that 
    $$\gamma_2=\gamma_j(r_\tau).$$
    This gives the equality
    $$ r_\tau=\frac{a+bq^k(v(a_i)-v(a_j))}{bq^{k+j}}.$$
    Since $a$ is coprime to $q$, we get that $a+bq^k(v(a_i)-v(a_j))$ is coprime to $q$ as well. This contradicts the assumption on $q$. Hence, $\zeta t^r$ is not an approximation term for $x_{<r}$ with respect to $P$ and $\lvert J(r)\rvert>1$ by Lemma \ref{lem: necessary approximation condition}.
    \end{proof}

Given $m\in\N$, we define $\Gamma_m:=\frac{1}{mp^\infty}\Z$.
\begin{corollary}
\label{cor: finite ramification away from p}
Let $F$ be a field of characteristic $p$. There is an algorithm $\maxram$ which takes as input a polynomial $f(X)\in\F_p(t)[X]$ and outputs a natural number $m$ not divisible by $p$ such that any root of $f$ in $\F((t^\Q))$ is already in $\F((t^{\Gamma_m}))$.
\end{corollary}
    
\begin{proof}
Let $\deg(f)=n$. Let $P(X)=\addpol(f)$, as in Remark \ref{rem: ore alg}.
Suppose that $x\in \F((t^\Q))$ is a root of $f(X)$, so in particular a root of $P(X)$. Let $S$ be the points of intersection of $P$. Let $q_1^{e_1},\ldots,q_\ell^{e_\ell}$ be the prime powers coprime to $p$ occurring as factors in the denominators of the reduced fractions of elements in $S$. By Theorem \ref{thm: ramifies at intersection} and by definition of a prime ramifying in $x$, any $r=a/b$ on reduced form in the support of $x$ is such that the prime power factors of $b$ coprime to $p$ divides $m=\prod_{i=1}^\ell q_i^{e_i}$. In other words, the support of $x$ is contained in $\Gamma_m$. The result thus follows from having $\maxram$ returning $m$ on input $f$.
\end{proof}

\begin{remark}
    Rayner shows the existence of the bound in Corollary \ref{cor: finite ramification away from p} by showing that  $\bigcup_{m=1}^\infty\xoverline{\F}((t^{\Gamma_m}))$ is closed under Artin-Schreier extensions. An alternative proof was given by Poonen in \cite[Corollary 7]{MR1225257}, which uses the following argument. Let $S$ be the set of automorphisms on $\F((t^\Q))$ given by sending $\sum_{\gamma}a_{\gamma}t^{\gamma}$ to $\sum_{\gamma}\zeta(\gamma)a_{\gamma}t^{\gamma}$, where $\zeta$ is a homomorphism from $\Q/\Z$ to the roots of unity in $\F$. Then $\bigcup_{m=1}^\infty\xoverline{\F}((t^{\Gamma_m}))$ is the subfield of $\F((t^\Q))$ consisting of elements with finite orbits under the action of $S$. This is an algebraically closed field by \cite[Lemma 5]{MR1225257}. It was noted by Konstantinos Kartas that effectiveness also follows from the proof by Poonen, since high ramification away from $p$ gives too many roots of $f$.
    \end{remark}
    
    We obtain the following analogue of Rayner's result for the residue field.
    \begin{theorem}
    \label{thm: expanding bounded}
     Let $\F$ be a perfect field of characteristic $p$ and let $P(X)\in\F_p(t)[X]$
    be an additive polynomial. Let $\zeta\in \Fbar$. If $x\in\F((t^\Q))$ is a root of $P$ and $r\in\Q$ is such that $\zeta$ expands $x$ at $r$, then $r$ is a point of intersection of $P$.
    \end{theorem}
    
    \begin{proof}
    Write $w=x_{<r}$. If $\lvert J(r)\rvert=1$, then $\zeta t^r$ is an approximation term for $x_{<r}$ and $\zeta$ satisfies an equation of the form $\zeta^{p^i}=b$, where $b\in K^*$. Since $K$ is an algebraic extension of $\F$, it is perfect and so $\zeta\in K$ This contradicts the assumption that $\zeta$ expands $x$ at $r$. Hence, $\lvert J(r)\rvert>1$, and $r$ is a point of intersection of $P$.
    \end{proof}
    
\begin{corollary}
    \label{cor: finite expansion away from p roots}
There is an algorithm $\maxexp$ which takes as input a polynomial $f(X)\in\F_p(t)[X]$ and outputs a natural number $m$ such that any root of $f$ in $\FbarQ$ is already in $\F_{p^m}((t^\Q))$.
\end{corollary}

\begin{proof}
Let $\deg(f)=n$ and let $P(X)=\addpol(f)$. Write $P(X)=\sum_{i=0}^na_iX^{p^i}$. Suppose that $x\in \FbarQ$ is a root of $f(X)$, so in particular a root of $P(X)$. Let $r$ be in the support of $x$ and let $w=x_{<r}$. Suppose that $w\in \F_{p^u}((t^\Q))$ for some $u\in \N$.
Let $\zeta t^r$ be the initial term of $x-w$. Suppose that $\lvert J(r)\rvert = 1$. Then $\zeta$ satisfies an equation of the form
$$\zeta^{p^\ell}=a$$
where $a\in \F_{p^u}$. Since $\F_{p^u}$ is perfect, we get that $\zeta\in \F_{p^u}$. If $\lvert J(r)\rvert > 1$, then $r$ is a point of intersection of $P$ and $\zeta$ satisfies an equation of the form
$$\sum_{j\in J(r)}a_j\zeta^{p^j}=a$$
with $a\in \F_{p^u}$. This shows that $a$ has degree at most $p^k$ over $\F_{p^u}$, with $k$ being maximal such that $\gamma_k\in J(r)$. In other words, $a$ is contained in $\F_{p^{mk!}}$.

Since the initial term $\zeta_0 t^{v(x)}$ of $x$ is a better approximation of $x$ than $0$ and since $P(0)\in \F_p$, we get in particular that the $\zeta_0$ has degree at most $p^{k_0}$ over $\F_p$, with $k_0$ being maximal such that $\gamma_{k_0}\in J(v(x))$. This serves as the base case for concluding by transfinite induction that $x\in \F_{p^{u!}}((t^\Q))$, where $u=\prod_{i=1}^np^i$ over $\F_p$. Having $\maxexp$ returning $u!$ on input $f$ thus gives the desired result.
\end{proof}

\begin{remark}
Note that the bound obtained in the proof of Corollary \ref{cor: finite expansion away from p roots} cannot be sharp. It assumes that for every $k\in\{1,\ldots,n\}$, there is a point of intersection of $P$ such that $k$ is maximal among the $i$ such that $\gamma_i\in J(r)$. This is impossible since there cannot be $n$ points of intersection of $P$. It is possible to have $\maxram$ return a sharper bound, taking into account the points of intersection of $P$.
\end{remark}

\bibliographystyle{alpha}
\bibliography{biblio2}

\begin{thebibliography}{Lam86}

\bibitem[Abh56]{MR80647}
Shreeram Abhyankar.
\newblock Two notes on formal power series.
\newblock {\em Proc. Amer. Math. Soc.}, 7:903--905, 1956.

\bibitem[AS03]{MR1997038}
Jean-Paul Allouche and Jeffrey Shallit.
\newblock {\em Automatic sequences}.
\newblock Cambridge University Press, Cambridge, 2003.
\newblock Theory, applications, generalizations.

\bibitem[Eis95]{MR1322960}
David Eisenbud.
\newblock {\em Commutative algebra}, volume 150 of {\em Graduate Texts in
  Mathematics}.
\newblock Springer-Verlag, New York, 1995.
\newblock With a view toward algebraic geometry.

\bibitem[KL19]{MR3911738}
Julia~F. Knight and Karen Lange.
\newblock Lengths of developments in {$K((G))$}.
\newblock {\em Selecta Math. (N.S.)}, 25(1):Paper No. 14, 36, 2019.

\bibitem[Lam86]{MR846958}
David Lampert.
\newblock Algebraic {$p$}-adic expansions.
\newblock {\em J. Number Theory}, 23:279--284, 1986.

\bibitem[Lis21]{lisinski}
Victor Lisinski.
\newblock Decidability of positive characteristic tame {H}ahn fields in
  $\mathcal{L}_t$. {P}reprint, ar{X}iv:2108.04132, 2021.

\bibitem[Poo93]{MR1225257}
Bjorn Poonen.
\newblock Maximally complete fields.
\newblock {\em Enseign. Math. (2)}, 39(1-2):87--106, 1993.

\bibitem[Ray68]{MR234941}
Francis~J. Rayner.
\newblock An algebraically closed field.
\newblock {\em Glasgow Math. J.}, 9:146--151, 1968.

\end{thebibliography}

\Addresses

\end{document}